\RequirePackage{xcolor}
\documentclass[journal,twoside,web]{cls/ieeecolor}  

\let\labelindent\relax

\usepackage{generic}
\usepackage{lcsys}
\usepackage[noadjust]{cite}
\usepackage{enumitem}
\usepackage{amsmath,amssymb,amsfonts,amsthm,dsfont} 
\usepackage{mathtools}
\usepackage{verbatim}
\usepackage[normalem]{ulem}

\usepackage{graphicx}
\graphicspath{{fig/}}
\usepackage[font={small}]{caption}
\usepackage{subcaption}
\usepackage[breaklinks=true, colorlinks, bookmarks=true, citecolor=black, urlcolor=violet,linkcolor=black]{hyperref}

\newtheorem{theorem}{Theorem}

\newtheorem{proposition}{Proposition}

\theoremstyle{definition}
\newtheorem{definition}{Definition}
\newtheorem{assumption}{Assumption}

\newtheorem*{problem*}{Problem}
\theoremstyle{remark}
\newtheorem*{remark}{Remark}

\newcommand{\ba}{\begin{align}}
\newcommand{\ea}{\end{align}}

\newcommand{\R}{{\mathbb R}}

\newcommand{\TODO}[1]{{\color{red}#1}}

\DeclareMathOperator*{\rank}{rank}

\newcommand{\NA}[1]{$\spadesuit$\footnote{\color{red}{Nikolay}: #1}}

\newcommand{\calB}{{\cal B}}
\newcommand{\calC}{{\cal C}}
\newcommand{\calD}{{\cal D}}
\newcommand{\calE}{{\cal E}}

\newcommand{\calH}{{\cal H}}

\newcommand{\calK}{{\cal K}}
\newcommand{\calL}{{\cal L}}

\newcommand{\calQ}{{\cal Q}}

\newcommand{\calS}{{\cal S}}

\newcommand{\calU}{{\cal U}}
\newcommand{\calV}{{\cal V}}

\newcommand{\calX}{{\cal X}}

\newcommand{\bfa}{\mathbf{a}}
\newcommand{\bfb}{\mathbf{b}}

\newcommand{\bfe}{\mathbf{e}}
\newcommand{\bff}{\mathbf{f}}

\newcommand{\bfq}{\mathbf{q}}
\newcommand{\bfr}{\mathbf{r}}

\newcommand{\bfu}{\mathbf{u}}

\newcommand{\bfw}{\mathbf{w}}
\newcommand{\bfx}{\mathbf{x}}
\newcommand{\bfy}{\mathbf{y}}
\newcommand{\bfz}{\mathbf{z}}

\newcommand{\bfA}{\mathbf{A}}

\newcommand{\bfC}{\mathbf{C}}

\newcommand{\bfF}{\mathbf{F}}
\newcommand{\bfG}{\mathbf{G}}

\newcommand{\bfN}{\mathbf{N}}

\newcommand{\bbR}{\mathbb{R}}

\title{\LARGE \bf Control Synthesis for Stability and Safety by Differential Complementarity Problem}

\author{Yinzhuang Yi, Shumon Koga, Bogdan Gavrea, and Nikolay Atanasov
	\thanks{We gratefully acknowledge support from NSF RI IIS-2007141.}%
	\thanks{Y. Yi, S. Koga, and N. Atanasov are with the Department of Electrical and Computer Engineering, UC San Diego, 9500 Gilman Drive, La Jolla, CA, 92093, USA (e-mails: {\tt \{yiyi,skoga,natanasov\}@ucsd.edu}).}%
	\thanks{B. Gavrea is with the Department of Mathematics, Technical University of Cluj-Napoca, Cluj-Napoca, 400114, Romania (e-mail: {\tt bogdan.gavrea@math.utcluj.ro}).}
}

\begin{document}
    \def\BibTeX{{\rm B\kern-.05em{\sc i\kern-.025em b}\kern-.08em
    T\kern-.1667em\lower.7ex\hbox{E}\kern-.125emX}}
    \markboth{\journalname, VOL. XX, NO. XX, XXXX 2017}
    {YI \MakeLowercase{\textit{et al.}}: Control Synthesis for Stability and Safety by Differential Complementarity Problem}
	\maketitle
	\thispagestyle{empty}
	\begin{abstract}
	This paper develops a novel control synthesis method for safe stabilization of control-affine systems as a Differential Complementarity Problem (DCP). Our design uses a control Lyapunov function (CLF) and a control barrier function (CBF) to define complementarity constraints in the DCP formulation to certify stability and safety, respectively. The CLF-CBF-DCP controller imposes stability as a soft constraint, which is automatically relaxed when the safety constraint is active, without the need for parameter tuning or optimization. We study the closed-loop system behavior with the CLF-CBF-DCP controller and identify conditions on the existence of local equilibria. Although in certain cases the controller yields undesirable local equilibria, those can be confined to a small subset of the safe set boundary by proper choice of the control parameters. Then, our method can avoid undesirable equilibria that CLF-CBF quadratic programming techniques encounter.
	\end{abstract}


\section{Introduction}

\IEEEPARstart{S}{tability} verification and stabilizing control design are fundamental problems in control theory that have impacted numerous industrial systems. One of the main tools for constructing control laws that stabilize nonlinear systems is a control Lyapunov function (CLF)~\cite{arstein,sontag1989universal}. As control systems are increasingly deployed in less structured, yet safety-critical settings, in addition to stability, control designs need to guarantee safety. Inspired by the property of CLFs to yield invariant level sets, control barrier functions (CBFs) \cite{ames2017clfcbfqp} have been developed to enforce that a desired safe subset of the state space is invariant. While various techniques for guaranteeing safety and stability exist, important challenges remain when both requirements are considered simultaneously. They include conditions under which it is possible to obtain a single control policy that guarantees CLF stability and CBF safety simultaneously (compatibility \cite{pol2022roa}), existence and uniqueness of the closed-loop system trajectories (well-posedness), convergence to points other than the origin (undesired equilibria), identification of initial conditions that ensure joint stability and safety (region of attraction).



\begin{figure}[t]
	\centering
	\includegraphics[width=\linewidth,trim={13mm 6mm 13mm 8mm}, clip]{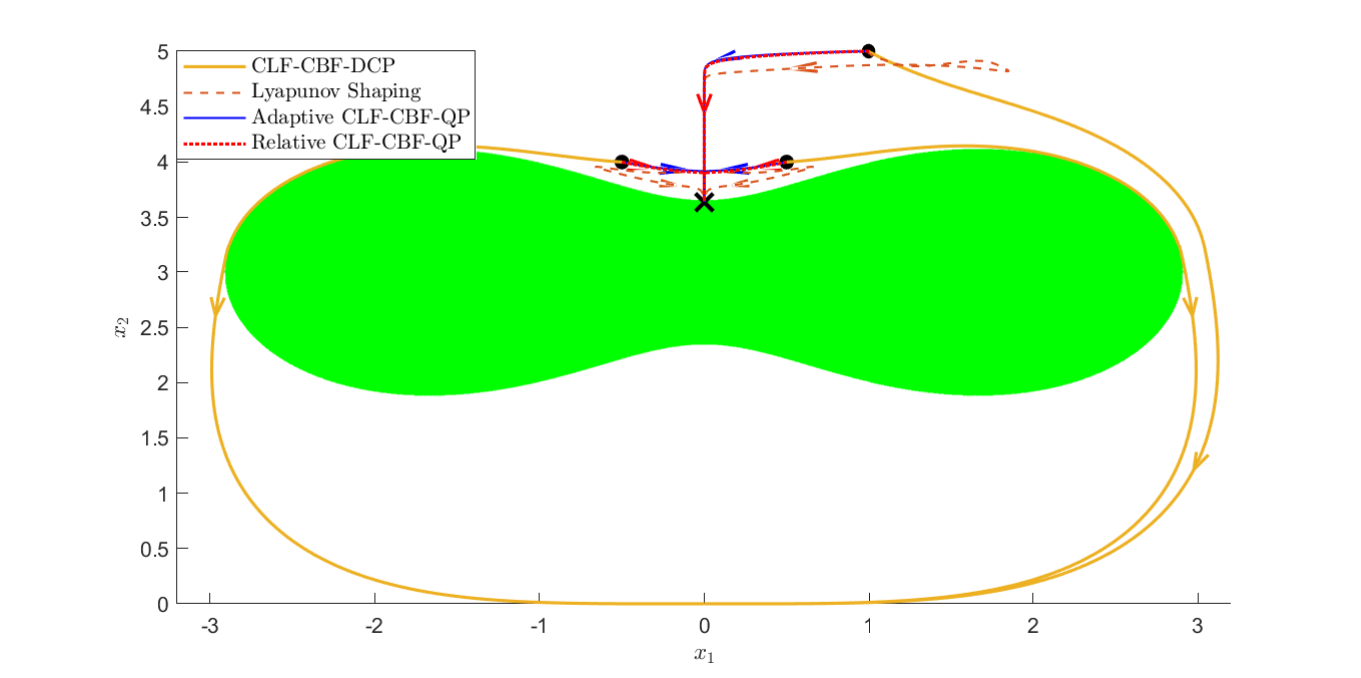}
	\caption{Comparison of the closed-loop trajectories resulting from our CLF-CBF-DCP control approach (\textbf{\textcolor{orange}{orange}}), two recent CLF-CBF-QP techniques
    (\textbf{\textcolor{blue}{blue}} \cite{pol2022roa} and \textbf{\textcolor{red}{red}} \cite{tan2021undesired}) and Lyapunov shaping (\textbf{\textcolor{brown}{brown}} \cite{reis2020control}) with a concave obstacle (\textbf{\textcolor{green}{green}}) for three initial conditions (\textbf{black} dots). Existing CLF-CBF-QP techniques converge to an undesired equilibrium on the obstacle boundary (\textbf{black} cross). Lyapunov shaping converges to a different undesired equilibrium on the obstacle boundary near the black cross for different initial conditions. Our approach stabilizes the system at the origin.}
	\label{fig: pre}
\end{figure}

The design of stabilizing control with guaranteed safety is studied in \cite{clbf}. The authors define a Control Lyapunov Barrier Function (CLBF), whose existence enables joint stabilization and safety. However, constructing a CLBF may be challenging and may require modification of the safe set. A comprehensive overview of CBF techniques and their use as safety constraints in quadratic programs (QPs) for control synthesis is provided in \cite{ames2017clfcbfqp}. Garg and Panagou~\cite{garg2021robust} propose a fixed-time CLF and analyze conditions for finite-time stabilization to a region of interest with safety guarantees. Local asymptotic stability for a particular CLF-CBF QP was proven in \cite{jankovic2018robust}. Reis et al.~\cite{reis2020control} show that a CLF-CBF-QP controller introduces equilibria other than the origin. The authors develop a new formulation by introducing a new parameter in the CLF constraint and an additional CBF constraint, aimed at avoiding undesired equilibria on the safe set boundary. In \cite{tan2021undesired}, it is shown that minimizing the distance to a nominal controller satisfying the CLF condition as the objective of a CLF-CBF QP ensures local stability of the origin without any additional assumptions. However, the region of attraction of this controller has not be characterized yet. Compatibility between CLF and CBF constraints is considered in \cite{shaw2022croa}. The authors define a CBF-stabilizable sublevel set of the CLF and characterize the conditions under which the closed-loop system is asymptotically stable with respect to the origin. Mestres and Cort{\'e}s \cite{pol2022roa} develop a penalty method to incorporate the CLF condition as a soft constraint in a QP, identify conditions on the penalty parameter that eliminate undesired equilibria, and provide an inner approximation of the region of attraction, where joint stability and safety is ensured. As illustrated in Fig.~\ref{fig: pre}, however, the regions of attraction guaranteed by recent CLF-CBF-QP techniques remain conservative, and initial conditions inside the safe set may still converge to undesired equilibria.

Our first contribution is a formulation of control synthesis with CLF stability and CBF safety constraints as a \emph{Differential Complementarity Problem} (DCP). A DCP is composed of an ODE subject to complementarity constraints:
\begin{equation}\label{eq:generalDCP}
\begin{aligned}
    & \dot{\bfy} = \hat{\bff}( \bfy,\bfz) \\
    \mathrm{s.t.} \, \, & 0 \leq \bfz \perp \bfC(\bfy,\bfz)\geq 0,
\end{aligned}
\end{equation}
where $\bfa\perp \bfb$ means that $\bfa$ and $\bfb$ are orthogonal vectors, i.e., $\bfa^{\top} \bfb = 0$. Complementarity problems are used to model combinatorial constraints in nonlinear optimization \cite{scheel2000mathematical} or hybrid dynamics in mechanical systems with contact \cite{posa2014direct}. We formulate a DCP in which the closed-loop dynamics are subject to constraints that certain control terms are activated only when the state may violate the CLF or CBF conditions. 


Our second contribution is an analysis of the closed-loop system equilibria, showing that our CLF-CBF-DCP controller can eliminate undesired equilibria on the safe set boundary that CLF-CBF-QP methods encounter (see Fig.~\ref{fig: pre}). We prove that our controller is Lipschitz continuous, ensuring existence and uniqueness of solutions, and show that the set of undesired equilibria can be restricted to a small subset of the safe set boundary by choosing a sufficiently large control gain parameter. A related work in the DCP literature is Camlibel et al. \cite{camlibel2007lyapunov}, which derives conditions for Lyapunov stability of linear DCP. While DCP formulations have potential to capture switching behavior in dynamical systems, they have not yet been used for safe control design. This paper introduces a new idea to model the mode switch between stability and safety satisfaction as a DCP.

\section{Preliminaries}



Consider a nonlinear control-affine system:
\begin{equation} \label{dynamics: Sys_dynamics}
	\dot{\bfx} = \bff(\bfx) + \bfG(\bfx)\bfu
\end{equation}
with state $\bfx \in \bbR^n$ and control input $\bfu \in \bbR^m$, where $\bff : \bbR^n \to \bbR^n$ and $\bfG : \bbR^n \to \bbR^{n \times m}$ are locally Lipschitz continuous. Assume that $\bff(\mathbf{0}) = \mathbf{0}$ so that the origin $\bfx = \bf0$ is the desired equilibrium of the unforced system. A typical objective is to design a stabilizing controller for the system. Stability of the closed-loop system is often verified by a Lyapunov function, while a stabilizing controller can be obtained using a control Lyapunov function.

%
%


\begin{definition} \label{CLF}
Given an open and connected set $\calD \subseteq \bbR^n$ with $\mathbf{0} \in \calD$, a continuously differentiable positive-definite function $V: \bbR^n \to \bbR$ is a \emph{control Lyapunov function} (CLF) on $\calD$ for system \eqref{dynamics: Sys_dynamics} if for each $\bfx \in \calD \setminus \{ \mathbf{0} \}$ it satisfies:
	\begin{equation} \label{criteria: CLF}
		\inf_{\bfu \in \bbR^m} [ \calL_{\bff} V(\bfx) + \calL_{\bfG} V(\bfx) \bfu ] \leq -\alpha_l(V(\bfx)),
	\end{equation}
where $\alpha_l: \bbR_{\geq 0} \to \bbR_{\geq 0}$ is a class $\calK$ function \cite{khalil1996nonlinear}.
\end{definition}

The set of stabilizing control inputs at state $\bfx \in \calD\setminus\{\mathbf{0}\}$, corresponding to a valid CLF $V$ on $\calD$, is:
\begin{equation*}
	K_{\textrm{clf}}(\bfx) = \{ \bfu \in \bbR^m: \calL_{\bff} V(\bfx) + \calL_{\bfG} V(\bfx) \bfu \leq -\alpha_l(V(\bfx)) \},
\end{equation*}
in the sense that any Lipschitz continuous control law $\bfr: \calD \mapsto \bbR^m$ such that $\bfr(\bfx) \in K_{\textrm{clf}}(\bfx)$ makes the origin of the closed-loop system asymptotically stable \cite{sontag2013mathematical}.

Beyond stability, it is often necessary to ensure that the system trajectories remain within a safe set $\calC \subset \calD$, in the sense that $\calC$ is forward invariant \cite{ames2019cbf}. We consider a closed set $\calC := \{\bfx \in \bbR^n | h(\bfx) \geq 0\}$, defined as the zero-superlevel set of a continuously differentiable function $h: \bbR^n \to \bbR$. One way to guarantee that $\calC$ is forward invariant is to require that $h$ is a control barrier function. Details can be found in the comprehensive work by Ames at el.~\cite{ames2019cbf}.

\begin{definition}\label{CBF}
A continuously differentiable function $h(\bfx): \bbR^n \to \bbR$ is a \emph{control barrier function} (CBF) of $\calC$ on $\calD$ for system \eqref{dynamics: Sys_dynamics} if it satisfies: 
	%
	\begin{equation} \label{criteria: CBF}
		\sup_{\bfu \in \bbR^m} [ \calL_{\bff} h(\bfx) + \calL_{\bfG} h(\bfx) \bfu ] \geq -\alpha_h(h(\bfx)), \ \forall \bfx \in \calD, 
	\end{equation}
	where $\alpha_h: \bbR \to \bbR$ is an extended class $\calK$ function.
\end{definition}

The set of safe control inputs at state $\bfx$, corresponding to a valid CBF $h$ on $\calD$, is:
\begin{equation}
	\nonumber
	K_{\textrm{cbf}}(\bfx) = \{ \bfu \in \bbR^m: \calL_{\bff} h(\bfx) + \calL_{\bfG} h(\bfx) \bfu \geq -\alpha_h(h(\bfx)) \},
\end{equation}
in the sense that any Lipschitz control law $\bfr: \calD \mapsto \bbR^m$ such that $\bfr(\bfx) \in K_{\textrm{cbf}}(\bfx)$ renders the set $\calC$ forward invariant \cite{xu2015robustness}.


Finding a single control input $\bfu$ achieving both stability and safety may be infeasible. In order to guarantee safety when \eqref{criteria: CLF} and \eqref{criteria: CBF} are not compatible \cite{universal-formula-safety}, a popular approach is to modify a stabilizing controller minimally so as to guarantee safety \cite{ames2019cbf}. Given a locally Lipschitz stabilizing control law $\bfr(\bfx) \in K_{\textrm{clf}}(\bfx)$, the following CBF QP obtains the minimum control perturbation to guarantee safety:
\begin{equation}\label{controller: CBF-QP}
	\begin{aligned}
		\min_{\bfu \in \bbR^m} &\,\| \bfu - \bfr(\bfx) \|_2^2  \\
		\mathrm{s.t.} \, &\, \calL_{\bff} h(\bfx) + \calL_{\bfG} h(\bfx) \bfu \geq -\alpha_h(h(\bfx)).
	\end{aligned}
\end{equation}

\section{Problem Statement}\label{sec: prob}

In agreement with the results in \cite{reis2020control,shaw2022croa,pol2022roa,tan2021undesired}, we note that a control law synthesized by the CBF QP in \eqref{controller: CBF-QP} introduces undesired local equilibria on the boundary of the safe set $\calC$ for the system in \eqref{dynamics: Sys_dynamics}. Hence, our objective is to design a new control synthesis method, which simultaneously guarantees safety and prevents undesired local equilibria for the closed-loop system.

Suppose the assumption below holds throughout the paper.

\begin{assumption}\label{assumption:cbf}
The safety requirements for system \eqref{dynamics: Sys_dynamics} are specified by a set $\calC = \{\bfx \in \bbR^n | h(\bfx) \geq 0\} \subset \calD$ such that $h$ is a CBF of $\calC$ with $h(\mathbf{0}) > 0$ and $\calL_{\bfG} h(\bfx) \neq \mathbf{0}$, $\forall \bfx \in \calD$.
\end{assumption}

Assumption~\ref{assumption:cbf} requires that the safe set $\calC$ is defined by a CBF with relative degree $1$ and the origin is in the interior of $\calC$. Under this assumption, the CBF-QP control law in \eqref{controller: CBF-QP} can be obtained in closed-form using the KKT conditions.

\begin{proposition}[{\cite[Thm.~1]{shaw2022croa}}]\label{sol: QP}
Consider system \eqref{dynamics: Sys_dynamics} with CLF $V$ on $\calD$ and CBF $h$ that satisfies Assumption \ref{assumption:cbf}. Then, the CBF QP in \eqref{controller: CBF-QP} has a unique closed-form solution:
\begin{equation}\label{sol: CBF-QP}
	\bfu^*(\bfx) = \left\{
	\begin{aligned}
		&\bfr(\bfx), \ &\nabla h(\bfx)^\top \bfe(\bfx) + \alpha_h(h(\bfx)) \geq 0, \\
		&\bar{\bfu}(\bfx), \ &\nabla h(\bfx)^\top \bfe(\bfx) + \alpha_h(h(\bfx)) < 0,
	\end{aligned} \right.
\end{equation}
where $\bfe(\bfx) := \bff(\bfx) + \bfG(\bfx)\bfr(\bfx)$ and
\begin{equation}
	\bar{\bfu}(\bfx) = \bfr(\bfx) - \frac{\nabla h(\bfx)^\top \bfe(\bfx) + \alpha_h(h(\bfx))}{\| \calL_{\bfG} h(\bfx) \|_2^2} \calL_{\bfG} h(\bfx)^\top.
\end{equation}
\end{proposition}

While the control law $\bfu^*(\bfx)$ in \eqref{sol: CBF-QP} guarantees safety by design, it might not asymptotically stabilize the system in \eqref{dynamics: Sys_dynamics} to the origin. There exist other local equilibria for the closed-loop system, characterized by the following proposition.

\begin{proposition} \label{prop: local_equi}
Consider the system in \eqref{dynamics: Sys_dynamics} with control law $\bfu^*(\bfx)$ in \eqref{sol: CBF-QP}. The equilibria of the closed-loop system are $\mathbf{0}$ or points $\bfx^* \in \partial \calC$ such that $\nabla h(\bfx^*)^\top \bfe(\bfx^*) < 0$ and 
\begin{equation}\label{criteria: LE-QP}
\bfe(\bfx^*) = \frac{\nabla h(\bfx^*)^\top \bfe(\bfx^*)}{\| \calL_{\bfG} h(\bfx^*) \|_2^2} \bfG(\bfx^*) \calL_{\bfG} h(\bfx^*)^{\top}.
\end{equation}
\end{proposition}

\begin{proof}
The result follows by replacing $\frac{1}{\epsilon} \calL_{g}V(x)$ in \cite[Propostion 5.1]{pol2022roa} with $-\bfr(\bfx)$.
\end{proof}

We focus on designing an alternative control law to \eqref{sol: CBF-QP}, which guarantees safety and eliminates the undesired local equilibria noted in Proposition~\ref{prop: local_equi}. 

\begin{problem*}
Consider the system in \eqref{dynamics: Sys_dynamics} with given CLF $V$ on $\calD$ and safe set $\calC = \{\bfx \in \bbR^n | h(\bfx) \geq 0\} \subset \calD$, defined by a CBF $h$. Design a control law that asymptotically stabilizes \eqref{dynamics: Sys_dynamics} to the origin while guaranteeing that $\calC$ is forward invariant.
\end{problem*}

	\section{Differential Complementarity Problem}

Our key idea is to introduce the CLF and CBF conditions as complementarity constraints in a DCP. We first review results in DCP theory, which will serve as the foundation of our CLF-CBF-DCP method.

We say that the DCP in \eqref{eq:generalDCP} has index $0$ if it is equivalent to a system of ODEs. Recall that $\bfC(\bfy,\bfz)$ is strongly monotone in $\bfz$, if there exists $c>0$, such that: 
\begin{equation}
    \left|\left|\bfC(\bfy, \bfz_1) - \bfC(\bfy, \bfz_2)\right|\right| \ge 
c\left|\left|\bfz_1 -\bfz_2\right|\right|, \;\; \forall \bfz_1,\bfz_2.
\end{equation}
If $\bfC(\bfy,\bfz)$ is Lispchitz in $\bfy$ and strongly monotone in $\bfz$, the DCP in \eqref{eq:generalDCP} has index $0$, $\bfz$ can be obtained as a Lipschitz continuous function of $\bfy$, and the DCP problem can be converted to an ODE with Lipschitz right-hand function \cite{STEWART2008812}.


Consider the case when $\bfC(\bfy,\bfz)$ is an affine function of $\bfz$. The corresponding initial-value DCP takes the form:
\begin{equation}\label{criteria: index-zero_DCP}
\begin{aligned}
    & \dot{\bfy} = \hat{\bff}(\bfy,\bfz) \\
    \mathrm{s.t.} \, \, & \mathbf{0} \leq \bfz \perp \bfA(\bfy) \bfz + \bfq(\bfy) \geq \mathbf{0},\\
	& \bfy(0) = \bfy_0,
\end{aligned}
\end{equation}
where the inequalities are applied element-wise. Denote the complementarity constraints of \eqref{criteria: index-zero_DCP} by $\mathrm{LCP}(\bfq(\bfy),\bfA(\bfy) )$, where $\mathrm{LCP}$ stands for linear complementarity problem. We are interested in assumptions on $\bfA(\bfy)$ under which $\mathrm{LCP}(\bfq(\bfy),\bfA(\bfy) )$ has a unique solution $\bfz(\bfy)$ for any choice of $\bfq(\bfy)$. The class of matrices for which this holds are $P$ matrices \cite[Thm.~3.3.7]{PangLCP}. A matrix is said to be a $P$ matrix when all of its principle minors are positive.

The next result considers a particular form for the matrix $\bfA(\bfy)$, which we will use to specify safety and stability constraints, and gives sufficient conditions for the DCP \eqref{criteria: index-zero_DCP} to be equivalent to a Lipschitz ODE. 

\begin{proposition} \label{prop: DCP_uniqueness}
    Consider the DCP in \eqref{criteria: index-zero_DCP}. Assume that the functions $\bfA(\bfy)$, $\bfq(\bfy)$ and $\bff(\bfy,\bfz)$ are locally Lipschitz continuous at $\bfy_0$ and that $\bfA(\bfy)$ has the form:
	\begin{equation}\label{eq:Amatrix}
	\bfA(\bfy) = \left[ \begin{matrix}
			a(\bfy) & 0\\
			c(\bfy)\ & d(\bfy)
		\end{matrix} \right] \in \R^{2\times 2},
	\end{equation}
	with elements satisfying either:
	\begin{enumerate}
	    \item[i)] $a(\bfy_0)>0$ and $d(\bfy_0)>0$, or
	    \item[ii)] $a(\bfy_0) = 0$, $d(\bfy_0) > 0$, and $q_1(\bfy_0) > 0$. 
	\end{enumerate}
	If either i) or ii) holds, then in a neighborhood of $\bfy_0$ the solution $\bfz(\bfy)$ of $\mathrm{LCP}(\bfq(\bfy),\bfA(\bfy))$ is unique, locally Lipschitz continuous, and has a closed-form expression:
    \begin{equation} 
    	\label{sol: LCP}
    	\begin{aligned}
    		&z_1(\bfy) = \begin{cases}
    			\frac{(-q_1(\bfy))_+}{a(\bfy)}, & a(\bfy) > 0, \\
    			0, & a(\bfy) = 0,
    			\end{cases}\\
    		&z_2(\bfy) = \frac{(-c(\bfy)z_1(\bfy)-q_2(\bfy))_+}{d(\bfy)},
    	\end{aligned}
    \end{equation}
    where $(\cdot)_+$ denotes $\max(\cdot,0)$. Hence, DCP \eqref{criteria: index-zero_DCP} is equivalent to an ODE with a locally Lipschitz right-hand side. 
\end{proposition}


\begin{proof}
    We prove the result under assumption i) first. From \eqref{eq:Amatrix} and the continuity assumptions on the problem data, it follows that $\bfA(\bfy)$ is a $P$ matrix in a neighborhood $\mathcal{V}$ of $\bfy_0$. This implies that the solution $\bfz(\bfy)$ of $\mathrm{LCP}(\bfq(\bfy),\bfA(\bfy))$ is unique for all $\bfy \in \mathcal{V}$. To derive the closed-form solution, consider the scalar $\mathrm{LCP}(\widetilde{b}, \widetilde{a})$: $0 \le t \perp \widetilde{a} t+\widetilde{b} \ge 0$ with $\widetilde{a}>0$. The scalar LCP has a unique solution: $t= \frac{(-\widetilde{b})_+}{\widetilde{a}}$. Due to the triangular form of the matrix $\bfA(\bfy)$, we can determine $z_1(\bfy)$ first by solving $\mathrm{LCP}(q_1(\bfy), a(\bfy))$. This gives $z_1(\bfy) = \frac{(-q_1(\bfy))_+}{a(\bfy)}$ for $\bfy \in \calV$. Substituting $z_1(\bfy)$ in the second complementarity constraint leads to another scalar $\mathrm{LCP}(c(\bfy) z_1(\bfy) +q_2(\bfy), d(\bfy))$ with $d(\bfy) > 0$ for $\bfy \in \calC$. Thus, $z_2(\bfy) = \frac{(-c(\bfy)z_1(\bfy)-q_2(\bfy))_+}{d(\bfy)}$. To prove that $\bfz(\bfy)$ is locally Lipshitz, we note that since $\bfA(\bfy)$ is a $P$ matrix in $\mathcal{V}$, we must have  $a(\bfy), d(\bfy) > 0$ for all $\bfy \in \calV$ which implies that $1/a(\bfy)$, $1/d(\bfy)$ are locally Lipschitz. The function $(r)_+$ with $r \in \bbR$ is locally Lipschitz and the composition $(f_2(x))_+$ satisfies the same property, whenever $f_2(x)$ is locally Lipschitz. Since the sum or product of two locally Lipschitz continuous functions are each locally Lipschitz, we can conclude that $\bfz(\bfy)$ is locally Lipschitz. This completes the proof under assumption i).

    Now, we consider assumption ii). Due to the continuity assumptions, we have $q_1(\bfy) > 0$, for all $\bfy$ in a neighborhood $\tilde{\calV}$ of $\bfy_0$. To satisfy the complementarity constraint, we need $z_1(\bfy)q_1(\bfy) = 0$. Since $q_1(\bfy) >0$ for all $\bfy\in \tilde{\calV}$ it follows that $z_1(\bfy) = 0$, $\bfy\in \tilde{\calV}$ and therefore $z_1(\bfy)$ is locally Lipschitz. The other component of $\bfz(\bfy)$, $z_2(\bfy)$, is the solution to the scalar $\mathrm{LCP}(q_2(\bfy), d(\bfy))$ with $d(\bfy) > 0$ and we can use the result under assumption i) to conclude that $z_2(\bfy)$ is Lipschitz in a neighborhood of $\bfy_0$.\qedhere
\end{proof}

	\section{CLF-CBF-DCP Control Design}
This section presents our DCP formulation for control synthesis with CLF stability and CBF safety constraints.
\subsection{DCP Formulation and Uniqueness of Solutions} 
Consider the system in \eqref{dynamics: Sys_dynamics} and parameterize its input as:
\begin{equation}
	\label{bfu: struct}
	\bfu = \bfw_l\bar{u}_l + \bfw_h\bar{u}_h.
\end{equation}
We will use the vectors $\bfw_l, \bfw_h \in \bbR^m$ to determine the directions of stability and safety satisfaction and the scalars $\bar{u}_l, \bar{u}_h \in \bbR$ to determine the control input magnitude. Consider the following DCP with CLF stability and CBF safety complementarity constraints:
\begin{subequations}
	\label{DCP: BQ-form}
	\begin{align}
		&\dot{\bfx} = \bff(\bfx) + \bfG(\bfx)[\bfw_l\bar{u}_l + \bfw_h\bar{u}_h] \label{DCP: ODE} \\
		\mathrm{s.t.} \, \, & 0 \leq \bar{u}_l \perp -F_l(\bfx) - \mathcal{L}_{\bfG}V(\bfx)  \bfw_l\bar{u}_l \geq 0, \label{DCP: CLC} \\
		& 0 \leq \bar{u}_h \perp F_h(\bfx) + \mathcal{L}_{\bfG}h(\bfx)  [\bfw_l\bar{u}_l + \bfw_h\bar{u}_h] \geq 0, \label{DCP: CBC}
	\end{align}
\end{subequations}
where $F_l(\bfx) := \mathcal{L}_{\bff}V(\bfx) + \alpha_l(V)$ and $F_h(\bfx) := \mathcal{L}_{\bff}h(\bfx) + \alpha_h(h)$. 
The constraint \eqref{DCP: CLC} requires that the magnitude term $\bar{u}_l$ responsible for ensuring stability is zero unless the CLF condition in Def.~\ref{CLF} is endangered. 
The safety constraint \eqref{DCP: CBC} needs to prevent safety violation caused by the stabilizing input $\bfw_l\bar{u}_l$. To achieve this, we add an overriding term $\bfw_h\bar{u}_h$ in \eqref{DCP: CBC}, which is zero unless the CBF condition in Def.~\ref{CBF} is endangered with control input $\bfw_l\bar{u}_l$. The system will switch from stabilizing mode to safe mode if $\bfw_h\bar{u}_h$ is not zero.
We show that under additional regularity assumptions, the control law resulting from the DCP in \eqref{DCP: BQ-form} is unique and locally Lipschitz continuous and, hence, ensures the existence and uniqueness of closed-loop system trajectories.

\begin{assumption} \label{assumption: lipschitz}
$\nabla V(\bfx)$ and $\nabla h(\bfx)$ are locally Lipschitz continuous for all $\bfx \in \calD \setminus \{ \mathbf{0} \}$. The functions $\alpha_l$ in Def.~\ref{CLF} and $\alpha_h$ in Def.~\ref{CBF} are locally Lipschitz continuous.
\end{assumption}

\begin{theorem} \label{trem: sol_DCP-BQ}
Assume $\calL_{\bfG}V(\bfx) \neq \mathbf{0}$, $\forall \bfx \in \calE \coloneqq \{\bfx \in \calD \setminus \{\mathbf{0}\} \mid F_l(\bfx)=0 \}$. Assume $\bfw_l$ satisfies $-\calL_{\bfG}V(\bfx)  \bfw_l > 0$ for all $\bfx \in \calD \setminus\{\mathbf{0}\}$ such that $\calL_{\bfG}V(\bfx) \neq \mathbf{0}$. Assume $\bfw_h$ satisfies $\mathcal{L}_{\bfG}h(\bfx)  \bfw_h > 0$ for all $\bfx \in \calD \setminus \{ \mathbf{0} \}$. Then, for all $\bfx \in \calD \setminus \{ \mathbf{0} \}$, the solution of the LCP in \eqref{DCP: CLC}-\eqref{DCP: CBC} is unique and has a closed-form:
\begin{equation} 
	\label{sol: DCP-BQ}
	\begin{aligned}
		&\bar{u}_l = \begin{cases}
			-\frac{F_l(\bfx)}{\mathcal{L}_{\bfG}V(\bfx)  \bfw_l}, & F_l(\bfx) \geq 0, \\
			0, & F_l(\bfx) < 0,
			\end{cases}\\
		&\bar{u}_h = \left(-\frac{F_h(\bfx) + \mathcal{L}_{\bfG}h(\bfx)  \bfw_l \bar{u}_l }{\mathcal{L}_{\bfG}h(\bfx)  \bfw_h}\right)_+.
	\end{aligned}
\end{equation}
If Assumption~\ref{assumption: lipschitz} holds, then $\bar{u}_l$ and $\bar{u}_h$ are locally Lipschitz continuous for all $\bfx \in \calD \setminus \{ \mathbf{0}\}$.
\end{theorem}
\begin{proof}
    Rewrite \eqref{DCP: CLC}-\eqref{DCP: CBC} as follows $0 \leq \bfz \perp \bfA(\bfx)\bfz + \bfq(\bfx) \geq 0$ with $\bfq(\bfx) = [-F_l(\bfx), F_h(\bfx)]^{\top}$ and $\bfA(\bfx) = \left[ \begin{matrix}
		-\mathcal{L}_{\bfG}V(\bfx)  \bfw_l & 0\\
		\mathcal{L}_{\bfG}h(\bfx)  \bfw_l & \mathcal{L}_{\bfG}h(\bfx)  \bfw_h
	\end{matrix} \right]$
    to reflect \eqref{criteria: index-zero_DCP}. When $\calL_{\bfG}V(\bfx) \neq \mathbf{0}$, we have $a(\bfx) = -\mathcal{L}_{\bfG}V(\bfx)  \bfw_l > 0$ and $d(\bfx) = \mathcal{L}_{\bfG}h(\bfx)  \bfw_h > 0$, which satisfies assumption i) in Proposition~\ref{prop: DCP_uniqueness}.
    When $\calL_{\bfG}V(\bfx) = \mathbf{0}$, we must have $q_1(\bfx) = -F_l(\bfx) > 0$ by our assumption on $\calL_{\bfG}V(\bfx)$ and Def.~\ref{CLF}. Thus, the LCP in \eqref{DCP: CLC}-\eqref{DCP: CBC} satisfies assumption ii) in Proposition~\ref{prop: DCP_uniqueness}. In both cases, by Proposition~\ref{prop: DCP_uniqueness}, the solution $\bar{u}_l$, $\bar{u}_h$ of $\mathrm{LCP}(\bfq(\bfx), \bfA(\bfx))$ is unique, locally Lipschitz continuous, and available in closed-form.
\end{proof}

A \emph{na{\"i}ve choice} of $\bfw_l$ and $\bfw_h$ that satisfies the requirements of Thm.~\ref{trem: sol_DCP-BQ} is $\bfw_l = -\mathcal{L}_{\bfG}V(\bfx)^{\top}$ and $\bfw_h = \mathcal{L}_{\bfG}h(\bfx)^{\top}$. It is also possible to avoid the extra assumption $\calL_{\bfG}V(\bfx) \neq \mathbf{0}$, $\forall \bfx \in \calE$ in Thm.~\ref{trem: sol_DCP-BQ}. If both $\calL_{\bfG}V(\bfx) = \mathbf{0}$ and $F_l(\bfx) = 0$, then $\bar{u}_l$ is not unique but, since stability is not endangered, we can choose $\bar{u}_l = 0$ to obtain the CLF-CBF-DCP controller:
\begin{equation} 
    \label{alt: DCP-BQ}
	\begin{aligned}
		&\bar{u}_l = \begin{cases}
			-\frac{F_l(\bfx)}{\mathcal{L}_{\bfG}V(\bfx)  \bfw_l}, & F_l(\bfx) > 0, \\
			0, & F_l(\bfx) \leq 0,
			\end{cases}\\
		&\bar{u}_h = \left(-\frac{F_h(\bfx) + \mathcal{L}_{\bfG}h(\bfx)  \bfw_l \bar{u}_l }{\mathcal{L}_{\bfG}h(\bfx)  \bfw_h}\right)_+.
	\end{aligned}
\end{equation}
We propose an alternative choice of $\bfw_h$ for our control law and study the equilibria of the closed-loop system next.

\subsection{Control Modification to Remove Undesired Equilibria}

As noted in Proposition \ref{prop: local_equi}, QP-based methods introduce undesired equilibria satisfying condition \eqref{criteria: LE-QP}. We analyze the equilibria under the CLF-CBF-DCP controller in \eqref{alt: DCP-BQ} and attempt to remove undesired ones using the degrees of freedom allowed by $\bfw_l$ and $\bfw_h$ in \eqref{bfu: struct}. Since $\bfw_h$ multiplies $\calL_{\bfG}h(\bfx)$ in the CBF constraint, introducing a component from the null space of $\calL_{\bfG}h(\bfx)$ to $\bfw_h$ will not change the safety condition but may be used to modify the control input. This motivates a modification of the na{\"i}ve choice of $\bfw_h$:
%
\begin{equation}
	\label{sel: bfw}
	\bfw_l = -\mathcal{L}_{\bfG}V(\bfx)^{\top}, \quad \bfw_h = \mathcal{L}_{\bfG}h(\bfx)^{\top} + k \bfw_p(\bfx), 
\end{equation}
where $k \in \bbR$ and $\bfw_p(\bfx)$ is locally Lipschitz and satisfies:
%
\begin{align} \label{wp-cond} 
\mathcal{L}_{\bfG}h(\bfx)  \bfw_p(\bfx) = 0, \; \bfG(\bfx)\bfw_p(\bfx) \neq \mathbf{0}, \; \forall \bfx \!\in \!\calD \!\setminus\!\{ \mathbf{0} \}.
\end{align}
The former requirement in \eqref{wp-cond} ensures that $\bfw_p(\bfx)$ does not affect the safety condition, while the latter that $\bfw_p(\bfx)$ affects the closed-loop dynamics. A sufficient condition for the latter requirement that $\bfw_p(\bfx)$ is not in the null space of $\bfG(\bfx)$ is $\bfN(\bfx)^\top \bfw_p(\bfx) = 0$, where $\bfN(\bfx) \in \bbR^{m \times (m-\rank(\bfG(x)))}$ is a null-space basis. Hence, \eqref{wp-cond} is satisfied if the following system has a non-zero solution:
\begin{equation} \label{eq:assumption3}
\begin{bmatrix}
    \nabla h(\bfx)^\top \bfG(\bfx) \\ \bfN(\bfx)^\top
\end{bmatrix} \bfw_p(\bfx) = \mathbf{0}.
\end{equation}
Since the matrix in \eqref{eq:assumption3} has $m+1-\rank(\bfG(x))$ rows, a sufficient condition for the existence of non-zero $\bfw_p(\bfx)$ satisfying \eqref{wp-cond} is $\rank(\bfG(x)) \geq 2$.

\begin{assumption} \label{assumption-rkG}
$\bfG(\bfx) \in \bbR^{n \times m}$
satisfies $\rank(\bfG(\bfx)) \geq 2$.
\end{assumption}

To highlight the role of $\bfw_p(\bfx)$ in the closed-loop system and simplify the notation, define $\bfF_u(\bfx) := \bff(\bfx) + \bfG(\bfx)[-\mathcal{L}_{\bfG}V^{\top}(\bfx)\bar{u}_l +  \mathcal{L}_{\bfG}h^{\top}(\bfx) \bar{u}_h]$ with $\bar{u}_l$, $\bar{u}_h$ given in \eqref{alt: DCP-BQ}. Then, the closed-loop system becomes:
\begin{equation} \label{sys: cl}
    \dot{\bfx} = \bfF_u(\bfx) + \bar{u}_h k \bfG(\bfx) \bfw_p(\bfx).
\end{equation}
%

%
%
\begin{proposition}\label{le: no_int}
Consider system \eqref{dynamics: Sys_dynamics} with the CLF-CBF-DCP control law in \eqref{bfu: struct}, \eqref{alt: DCP-BQ}, and \eqref{sel: bfw}. There are no equilibria of the closed-loop system in $\textrm{Int}(\calC)$ except $\mathbf{0}$.
\end{proposition} 

\begin{proof}
Any non-zero equilibrium $\bfx^*$ of \eqref{sys: cl} satisfies:
\begin{align} \label{equilibria-cond}
-\bfF_u(\bfx^*) = \bar{u}_h(\bfx^*) k \bfG(\bfx^*) \bfw_p(\bfx^*).
\end{align}
Suppose that $\bfx^* \in \textrm{Int}(\calC) \setminus \{\mathbf{0}\}$ and consider two cases.


\textit{(i)}: If $\bar{u}_h = 0$, from \eqref{equilibria-cond}:
\begin{equation}
   -\bfF_u(\bfx^*) = -\bff(\bfx^*) + \bfG(\bfx^*) \calL_{\bfG} V(\bfx^*)^\top \bar{u}_l = \mathbf{0}.
\end{equation}
However, the CLF condition in \eqref{DCP: CLC} implies that $\nabla V(\bfx^*)^\top \bfF_u(\bfx^*) \leq - \alpha_l(V(\bfx^*))$. Since $\alpha_l(V(\bfx)) = 0$ only when $\bfx = \mathbf{0}$, the CLF condition yields $\bfF_u(\bfx) \neq \mathbf{0}$ for all $\bfx \in \calD \setminus \{\mathbf{0}\}$ and hence $\bfx^*$ cannot be an equilibrium.

\textit{(ii)}: If $\bar{u}_h > 0$, the CBF constraint in \eqref{DCP: CBC} implies:
\begin{equation}\label{eq:uh-positive}
    \calL_{\bfF_u}h(\bfx^*) + \bar{u}_h k \calL_{\bfG}h(\bfx^*)  \bfw_p(\bfx^*) + \alpha_h(h(\bfx^*)) = 0.
\end{equation}
From \eqref{wp-cond}, \eqref{eq:uh-positive} becomes $\nabla h(\bfx^*)^\top \bfF_u(\bfx^*) = - \alpha_h(h(\bfx^*))$. For all $\bfx \in\textrm{Int}(\calC)$, it holds that $h(\bfx) > 0$ and $\alpha_h(h(\bfx)) > 0$. As a result, $\nabla h(\bfx^*)^\top \bfF_u(\bfx^*) \neq 0$ but, from \eqref{wp-cond}, $\nabla h(\bfx^*)^\top \bfG(\bfx^*)\bfw_p(\bfx^*) = 0$, indicating that $\bfF_u(\bfx^*)$ and $\bfG(\bfx^*)\bfw_p(\bfx^*)$ are linearly independent. Hence, they cannot be collinear as required by \eqref{equilibria-cond}, and $\bfx^*$ cannot be an equilibrium. By contradiction, there are no equilibria of the closed-loop system in $\textrm{Int}(\calC)\setminus\{\mathbf{0}\}$.
\end{proof}

\begin{remark} Our CLF-CBF-DCP formulation guarantees safety because the control law $\bfu(\bfx)$ given by \eqref{bfu: struct}, \eqref{alt: DCP-BQ}, and \eqref{sel: bfw} satisfies $\bfu(\bfx) \in K_{cbf}(\bfx)$, $\forall \bfx \in \calD$, making $\calC$ invariant \cite{xu2015robustness}.\end{remark}

As shown in Proposition~\ref{le: no_int}, equilibria besides $\mathbf{0}$ occur only on $\partial \calC$ with $\bar{u}_h > 0$, which implies $\bfF_u(\bfx) \neq \mathbf{0}$ if \eqref{equilibria-cond} holds. To eliminate an undesired equilibrium $\bfx^* \in \partial\calC$, we consider how to choose $k \in \bbR$ such that \eqref{equilibria-cond} does not hold. Let $\bfa \parallel \bfb$ indicate that vectors $\bfa$ and $\bfb$ are linearly dependent and in the same direction. Define the following sets:
\begin{align*}
\Omega & \coloneqq \{ \bfx \in \partial \calC \mid -\bfF_u(\bfx) \parallel \bfG(\bfx)\bfw_p(\bfx), \;\bfF_u(\bfx) \neq \mathbf{0} \}, \\
\calX &\coloneqq \{ \bfx \in \Omega \mid \bar{u}_h > 0 \},
\end{align*}
where $\calX$ is the set inside which undesired equilibria might occur.
If $\calX \equiv \emptyset$, then it is clear that \eqref{equilibria-cond} does not hold and no undesired equilibria occur. Consider $\cal X \neq \emptyset$ hereafter. Since $\calX \subseteq \partial \calC \subset \calD$, the requirement on $\bfw_p(\bfx)$ in \eqref{wp-cond} implies that $\inf_{\bfx \in \calX}\| \bfG(\bfx) \bfw_p(\bfx) \|>0$.
If $\inf_{\bfx \in \calX} \bar{u}_h > 0$, then choosing
\begin{equation}
	\label{criteria: k}
	k > \frac{\sup_{\bfx \in \calX} \| \bfF_u(\bfx) \|}{ \inf_{\bfx \in \calX} \bar{u}_h \| \bfG(\bfx) \bfw_p \|}
\end{equation}
ensures that all equilibria satisfying \eqref{equilibria-cond} would be removed. However, such $k$ may not exist if $\inf_{\bfx \in \calX} \bar{u}_h = 0$, which may be attained for some $\bfx \in \calS \coloneqq \{ \bfx \in \Omega \mid \bar{u}_h = 0,\; \calB_{\epsilon}(\bfx) \cap \calX \neq \emptyset, \;\forall \epsilon > 0 \}$ where $\calB_{\epsilon}(\bfx)$ is a Euclidean ball at $\bfx$ with radius $\epsilon$. Instead, suppose that with user-specified $\nu>0$:
\begin{equation}
    \label{criteria: relax_k}
    k > \frac{\sup_{\bfx \in \calX} \| \bfF_u(\bfx) \|}{ \nu \inf_{\bfx \in \calX} \| \bfG(\bfx) \bfw_p \|}.
\end{equation}
%


%
\begin{proposition} \label{local_equilibria}
Consider system \eqref{dynamics: Sys_dynamics} with the CLF-CBF-DCP control law in \eqref{bfu: struct}, \eqref{alt: DCP-BQ}, \eqref{sel: bfw}, and parameter $k$ satisfying \eqref{criteria: relax_k}. Then, closed-loop system eqilibria $\bfx^* \in \calD \setminus\{\mathbf{0}\}$ exist only in $\calQ \coloneqq \{ \bfx \in \Omega \mid 0 < \bar{u}_h < \nu \} \subset \partial \calC$ and satisfy \eqref{equilibria-cond}.
\end{proposition}

\begin{proof}
For all $\bfx \in \calX$, \eqref{criteria: relax_k} implies $k \nu \| \bfG(\bfx) \bfw_p \| \geq k \nu \inf_{\bfx \in \calX} \| \bfG(\bfx) \bfw_p \| > \sup_{\bfx \in \calX} \| \bfF_u(\bfx) \| \geq \| \bfF_u(\bfx) \|$. Thus, for $\bfx \in \calX$ such that $\bar{u}_h \geq \nu$, the inequality $k \bar{u}_h \| \bfG(\bfx) \bfw_p \| > \| \bfF_u(\bfx) \|$ holds and implies that \eqref{equilibria-cond} does not hold. From this and Proposition~\ref{le: no_int}, non-zero eqilibria may only occur in $\calQ = \calX \setminus \{\bfx \in \Omega \mid \bar{u}_h \geq \nu\}$.
\end{proof}
\begin{remark}
Our method of using $\bfw_l$, $\bfw_h$, $\bfw_p$ to parameterize the degrees of freedom in the control law can be formulated as a quadratic program, with cost function $\frac{1}{2}(\bar{u}^2_l + \bar{u}^2_h)$ and constraints $-F_l(\bfx) - \mathcal{L}_{\bfG}V(\bfx)  \bfw_l\bar{u}_l \geq 0, \ F_h(\bfx) + \mathcal{L}_{\bfG}h(\bfx)  [\bfw_l\bar{u}_l + \bfw_h\bar{u}_h] \geq 0$.
\end{remark}
Proposition \ref{local_equilibria} suggests that the set $\calQ$ of potential equilibria on $\partial\calC$ can be shrunk towards a small set $\calS$, defined above \eqref{criteria: relax_k}, by increasing $k$ in \eqref{sel: bfw}.
Hence, our CLF-CBF-DCP controller can avoid all undesirable equilibria in $\textrm{Int}(\calC)$ and some on $\partial \calC$. If $\bfF_u(\bfx)$ and $\bfG(\bfx)\bfw_p(\bfx)$ are linearly independent for all $\bfx \in \calX$, then $\calQ$ is empty and no local equilibria exist on $\partial \calC$. Compared with \cite{tan2021undesired}, which only requires existence of valid CLF and CBF, our formulation has two additional assumptions: $\calL_{\bfG}h(\bfx) \neq \mathbf{0}$ (Assumption~\ref{assumption:cbf}) and $\rank(\bfG(\bfx)) \geq 2$ (Assumption~\ref{assumption-rkG}). Assumption~\ref{assumption:cbf} may be relaxed by considering cases $a(\bfy_0) > 0$, $d(\bfy_0) = 0$ and $a(\bfy_0)=d(\bfy_0) = 0$ in Proposition~\ref{prop: DCP_uniqueness}, which we leave for future work. It is not possible to relax Assumption~\ref{assumption-rkG} because our approach relies on the extra degrees of freedom to define a direction $\bfG(\bfx)\bfw_p(\bfx)$ avoiding undesired equilibria.

\section{Evaluation} 
\label{sec:evaluation}

We compare our \emph{CLF-CBF DCP} with the \emph{relative CLF-CBF QP} in \cite{tan2021undesired}, the \emph{adaptive CLF-CBF QP} in \cite{pol2022roa} and the \emph{Lyapunov shaping} method in \cite{reis2020control}. Consider three scenarios for a fully actuated system, $\dot{\bfx} = \bfx + \bfu$, with $\bfx,\bfu \in \bbR^2$.
\begin{itemize}[align=left,leftmargin=1em]
    \item[Case 1:] We use CBF $h(\bfx) \!=\! x_1^2 + (x_2 - 4)^2 \!- 4$ with $\alpha_h(h) = h$ and CLF $V(\bfx) = \frac{1}{2} (6x_1^2 + x_2^2)$ with $\alpha_l(V) = V$.
    
    
    \item[Case 2:] We use CBF $h(\bfx) = \big( (x_1-2)^2+(x_2-3)^2 \big)\big( (x_1+2)^2+(x_2-3)^2 \big) - 2.1^4$ with $\alpha_h(h) = h$ and CLF $V(\bfx) = \frac{1}{2} (x_1^2 + x_2^2)$ with $\alpha_l(V) = V$. 
    
    \item[Case 3:] We use CBF $h(\bfx) = \big( (x_1-2)^2+(x_2-3)^2 \big)\big( (x_1+2)^2+(x_2-3)^2 \big) - 2.1^4$ with $\alpha_h(h) = h$ and CLF $V(\bfx) = \frac{1}{2} (x_1^2 + 2x_2^2)$ with $\alpha_l(V) = V$.
\end{itemize}

\begin{figure*}[t]
	\centering
	\subcaptionbox{Case 1 \label{fig: cir_6_ck}}{\includegraphics[width=0.32\linewidth,trim=13mm 6.5mm 10mm 8mm, clip]{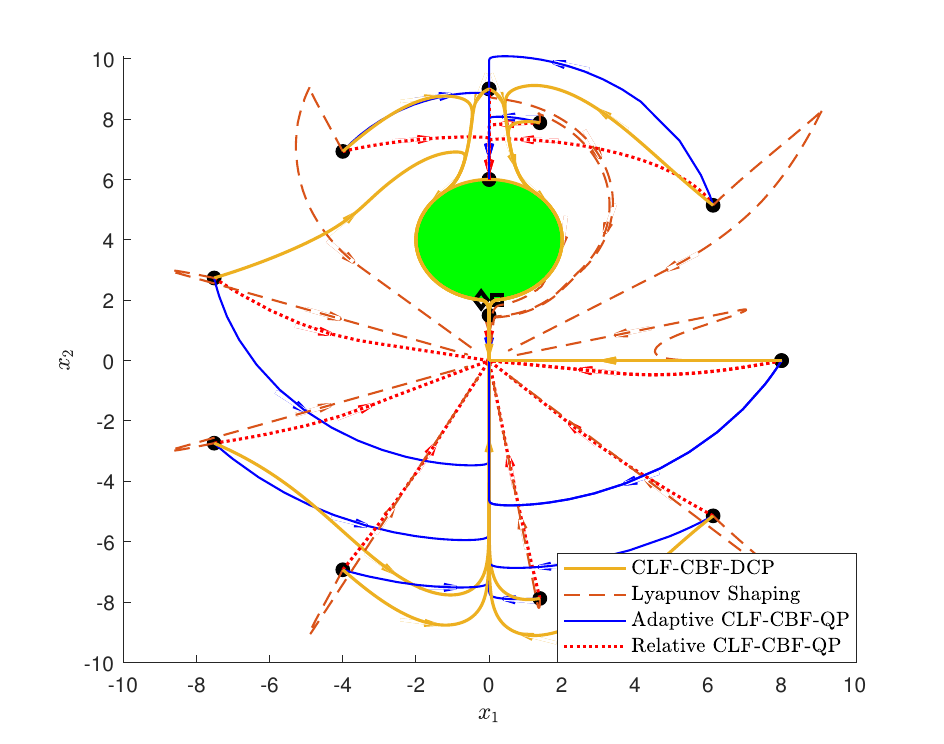}}%
	\hfill%
	\subcaptionbox{Case 2 \label{fig: con_ck}}{\includegraphics[width=0.32\linewidth,trim=13.5mm 6mm 10mm 8mm, clip]{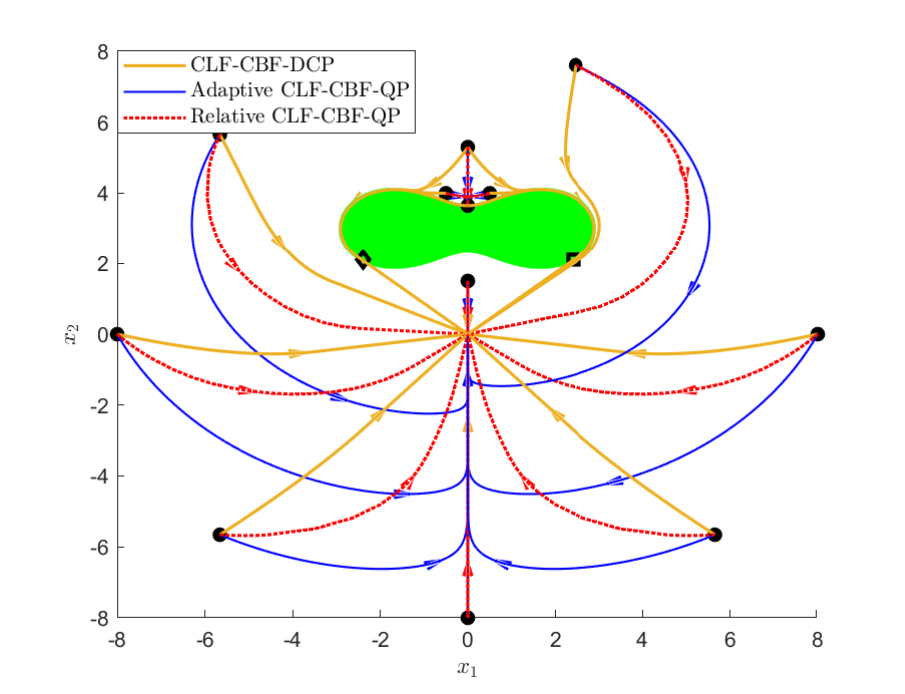}}%
	\hfill%
	\subcaptionbox{Case 3 \label{fig: con_nr}}{\includegraphics[width=0.32\linewidth,trim=12mm 5mm 10mm 7mm, clip]{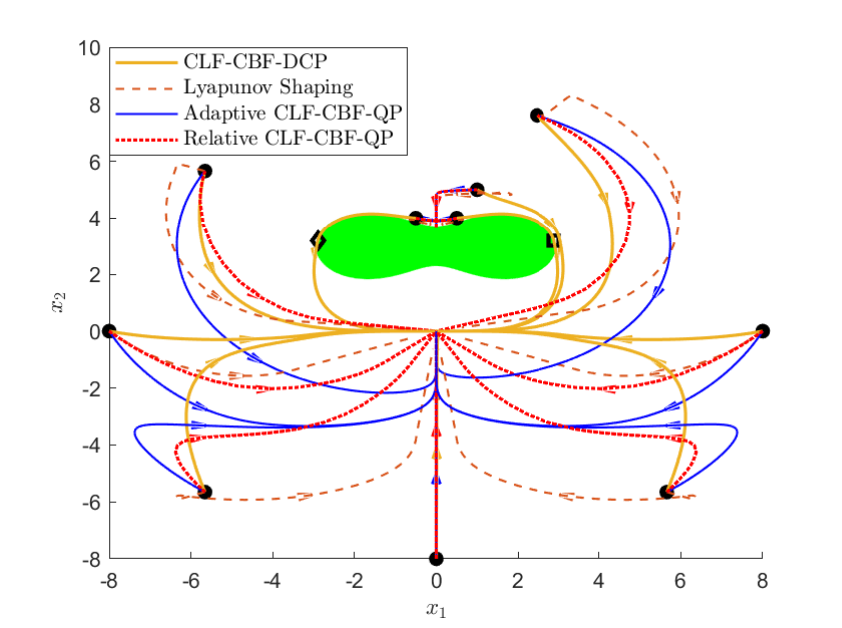}}%
    \caption{Comparison of our CLF-CBF-DCP controller to controllers obtained from the relative CLF-CBF-QP \cite{tan2021undesired}, adaptive CLF-CBF-QP \cite{pol2022roa}, and the Lyapunov shaping approach \cite{reis2020control} with various initial conditions (\textbf{black} dots) and obstacles (\textbf{\textcolor{green}{green}}) defined by the three cases in Sec.~\ref{sec:evaluation}. In Case 1, the relative and adaptive CLF-CBF-QP methods converge to an undesired local equilibrium $(0,6)$ on the safe set boundary with initial conditions satisfying $x_2(0) > 4$. In Case 2 and 3, these methods converge to the same boundary local equilibrium $(0,3.65)$ under initial conditions $(\pm0.5, 4)$.
    The Lyapunov shaping method avoids boundary equilibria for Case 1 but fails for Case 3, where it converges to boundary equilibria from initial conditions $(\pm0.5, 4)$. For all other initial conditions in Case 1 and 3, the Lyapunov shaping method converges to interior equilibria that are not the origin. The Lyapunov shaping method cannot be applied in Case 2 because the Lyapunov function is radial. In contrast, the CLF-CBF-DCP controller achieves both safety and stability at the origin in all three cases.}
	\label{fig: compare}
\end{figure*}

In each case, there are two choices of $\bfw_p$ satisfying \eqref{wp-cond} with opposite signs. In the simulations, for some interesting initial conditions, we show the closed-loop CLF-CBF-DCP trajectory for both choices. The trajectories of the four controllers are shown in Fig.~\ref{fig: compare}. The baseline methods converge to undesirable equilibria for some initial conditions, depending on the safe set shape. Our DCP controller achieves both safety and asymptotic stability at the origin for all tested initial conditions. The diamond and square points in Fig.~\ref{fig: compare} depict the locations of the set $\calS$, which is determined by $\bfw_p(\bfx)$. If the system trajectory passes through the \emph{diamond} point, then only the \emph{square} point belongs to $\calS$, and vice versa. 
Fig.~\ref{cir: k_test} shows the effect of the term $k$ in \eqref{sel: bfw} on the DCP controller. For small values of $k$ in Fig.~\ref{cir: k_test}(a), the trajectory gets stuck at an undesired local equilibrium, while for larger $k$, in Fig.~\ref{cir: k_test}(b), it converges to the origin. 


Our CLF-CBF DCP and the relative CLF-CBF QP \cite{tan2021undesired} introduce no local equilibria in $\textrm{Int}(\calC)$, while the adaptive CLF-CBF QP \cite{pol2022roa} needs to satisfy \cite[Corollary 5.4]{pol2022roa} to remove local equilibria in $\textrm{Int}(\calC)$. Fig.~\ref{fig: cir_6_ck} and Fig.~\ref{fig: con_nr} show that Lyapunov shaping \cite{reis2020control} fails to address local equilibria in $\textrm{Int}(\calC)$. The relative CLF-CBF QP can get stuck at local equilibria on $\partial \calC$ if $L_{\bfG}h(\bfx) \neq \mathbf{0}$ \cite[Theorem~3.(3)]{tan2021undesired} and so can the adaptive CLF-CBF QP if \cite[Remark 2]{pol2022roa} holds. Lyapunov shaping avoids local equilibria on $\partial \calC$ caused by collinearity among the vectors $f(\bfx)$, $\bfG(\bfx)\nabla V(\bfx,Q)$ and $\bfG(\bfx)\nabla h(\bfx)$, as shown in Fig.~\ref{fig: compare}. Our CLF-CBF-DCP formulation has a closed-form solution, which can be obtained in constant time and is suitable for real-time operation.

\begin{figure}[t]
    \vspace*{-1ex}
    \subcaptionbox{$k = 14.6$ at the \textit{plus} point\label{fig: k_test_fail}}{\includegraphics[width=0.5\linewidth,trim=13mm 7mm 12mm 8mm, clip]{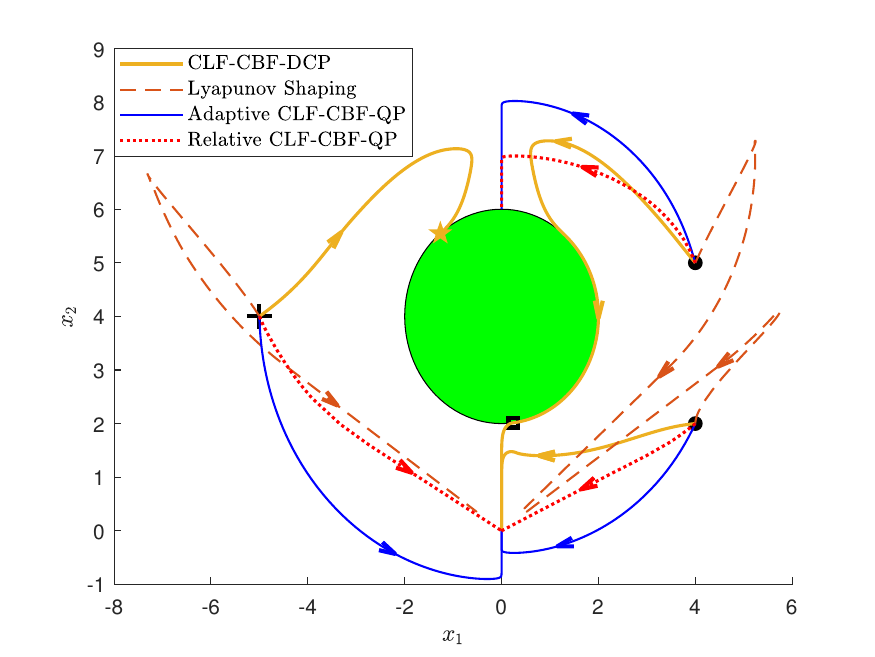}}%
    \hfill%
    \subcaptionbox{$k = 15$ at the \textit{plus} point\label{fig: k_test_succ}}{\includegraphics[width=0.5\linewidth,trim=13mm 7mm 12mm 8mm, clip]{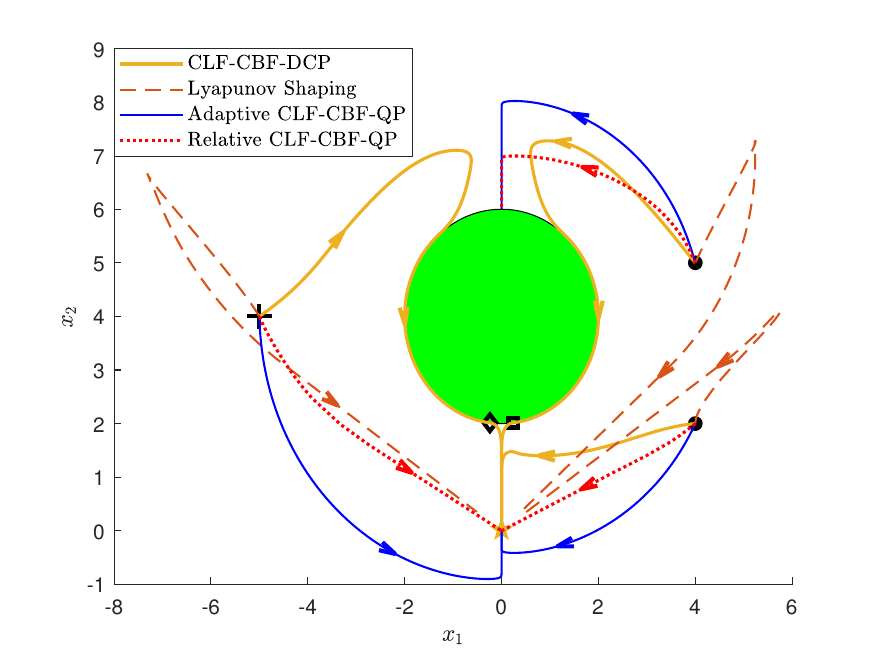}}%
    \caption{The CLF-CBF-DCP controller with different choices of $k$ initialized at the \textbf{black} plus point $(-5,4)$. Fig.~\ref{fig: k_test_fail} shows that if $k$ is not sufficiently large, a local equilibrium exists on the obstacle boundary at the \textbf{\textcolor{orange}{orange}} star point. Fig.~\ref{fig: k_test_succ} shows that this equilibrium can be eliminated by increasing $k$. }
    \label{cir: k_test}
\end{figure}

\section{Conclusion} 
We formulated control synthesis with CLF stability and CBF safety constraints as a DCP. We proved that the CLF-CBF-DCP control law is Lipschitz continuous and characterized the closed-loop system equilibria. We demonstrated how an extra forcing term may eliminate undesired equilibria that CLF-CBF-QP techniques encounter. 

	\bibliographystyle{cls/IEEEtran}
	\bibliography{ref/BIB_ICRA22.bib}
\end{document}